\documentclass[12pt,fleqn,leqno]{amsart}
\usepackage{amsfonts,amssymb}
\usepackage{mathrsfs}
\usepackage{enumitem}
\usepackage{bm}

\usepackage[svgnames]{xcolor}
\usepackage[colorlinks,linkcolor=blue,citecolor=Green]{hyperref}
\usepackage{titlesec}

\usepackage{tikz-cd}

\titleformat{\section}[block]
 {\bfseries}
 {\thesection.}
 {\fontdimen2\font}
 {}

\newcommand{\periodafter}[1]{#1.}

\titleformat{\subsection}[runin]
 {\bfseries}
 {\thesubsection.}
 {\fontdimen2\font}
 {\periodafter}

\setlist{topsep=1ex, noitemsep}
\setenumerate{labelindent=\parindent,label=\upshape{(\alph*)}}

\newtheorem{theorem}{Theorem}[section]
\newtheorem{corollary}[theorem]{Corollary}
\newtheorem{proposition}[theorem]{Proposition}

\renewcommand{\emptyset}{\varnothing}
\numberwithin{equation}{section}

\DeclareMathOperator{\uhr}{\upharpoonright} 
\DeclareMathOperator\ssmap{\leadsto}
\DeclareMathOperator{\R}{\mathbb{R}}
\DeclareMathOperator{\N}{\mathbb{N}}
\DeclareMathOperator{\coz}{coz}

\DeclareMathOperator{\diam}{diam}

\providecommand{\germ}{\mathfrak}

\begin{document}

\title{Paracompactness and Open Relations}

\author{Valentin Gutev}

\curraddr{Department of Mathematics, Faculty of Science, University of
   Malta, Msida MSD 2080, Malta}
\email{\href{mailto:gutev@fmi.uni-sofia.bg}{gutev@fmi.uni-sofia.bg}}

\subjclass[2010]{54B10, 54C35, 54C60, 54C65, 54D20.}

\keywords{Paracompactness, relation, set-valued mapping, insertion, selection.}

\begin{abstract}
  The countably paracompact normal spaces were characterised by Dowker
  and Kat\v{e}tov in terms of an insertion property. Dowker also
  described them by normality of their product with the closed unit
  interval. Michael used the Dowker-Kat\v{e}tov insertion property to
  motivate his selection characterisation of these spaces. Morita
  extended in a natural way Dowker's product characterisation to all
  $\tau$-paracompact normal spaces. In this paper, we look at these
  results from the point of view of open relations. Insertions and
  selections are equivalent for such relations. Furthermore, we obtain
  a natural characterisation of $\tau$-paracompact normal spaces in
  terms of selections for convex-valued open relations. Based on this,
  we give simple alternative proofs of the above mentioned
  results. Other applications are obtained as well.
\end{abstract}

\date{\today}
\maketitle

\section{Introduction}
\label{sec:introduction}

In this paper, all topological spaces are assumed to be Hausdorff,
$w(Y)$ is the \emph{topological weight} of a space $Y$ and $d(Y)$ is
its \emph{density}. We will consider only real vector spaces. A vector
space $E$ is a \emph{topological vector space} if it is endowed with a
topology with respect to which the vector operations are
continuous.\medskip

For
an infinite cardinal number $\tau$, a space $X$ is
\emph{$\tau$-paracompact} \cite{zbMATH03161751,morita:60} if each open
cover of $X$ of cardinality $\leq\tau$, has a locally finite open
refinement. If $\omega$ is the first infinite ordinal, then
$\omega$-paracompactness is nothing else but countable
paracompactness. By definition, a space $X$ is paracompact if it is
$\tau$-paracompact for any cardinal number $\tau$. Paracompactness of
$X$ implies normality of $X$. However, there are many examples of
countably paracompact spaces which are not normal. Also, there are
simple examples of $\tau$-paracompact spaces which are not
$\tau^+$-paracompact \cite{morita:60}, where the cardinal $\tau^+$ is
the immediate successor of $\tau$.\medskip

For a relation ${\Phi\subset X\times Y}$, the set
$\Phi(x)=\{y\in Y: \langle x,y\rangle\in \Phi\}$ is the \emph{image}
of a point $x\in X$ by $\Phi$. More generally, the \emph{image} of a
subset $A\subset X$ by $\Phi$ is the set
$\Phi[A]=\bigcup_{x\in A}\Phi(x)$. The \emph{inverse} $\Phi^{-1}$ of a
relation $\Phi\subset X\times Y$ is the relation
$\Phi^{-1}=\{\langle y,x\rangle: \langle x,y\rangle\in \Phi\}\subset
Y\times X$. In these terms, the image $\Phi^{-1}[B]$ of a subset
$B\subset Y$ by the inverse relation $\Phi^{-1}$ is the
\emph{preimage} of $B$ by the relation $\Phi$; similarly,
$\Phi^{-1}(y)$ is the preimage of a point $y\in Y$ by $\Phi$. If
$\Phi^{-1}[Y]=X$, then $\Phi$ represents a map from $X$ to the
\emph{nonempty} subsets of $Y$ and we will simply write
$\Phi:X\ssmap Y$ to designate this property. In Michael's selection
theory, such a relation $\Phi:X\ssmap Y$ is commonly called a
\emph{set-valued mapping} (also a \emph{multifunction}, or simply a
\emph{carrier} \cite{michael:56a}). Finally, let us recall that a map
$f:X\to Y$ is a \emph{selection} (or a \emph{single-valued selection})
for $\Phi:X\ssmap Y$ if $f(x)\in\Phi(x)$ for every $x\in X$.\medskip

In this paper, we will broadly use the term ``relation'' which offers
the flexibility to identify a set-valued mapping with its graph. This
natural convention has the benefit that any adjective which can be
applied to sets (e.g. ``open'' or ``closed'', etc.) can be applied
automatically to set-valued mappings. Another benefit is that all
covers of $X$ indexed by the elements of $Y$ are precisely the
relations $\Phi:X\ssmap Y$, see \cite{MR3673071}. In this
interpretation, the \emph{indexed cover} identified with
$\Phi:X\ssmap Y$ is simply the inverse relation
$Y\ni y\to \Phi^{-1}(y)\subset X$. In contrast to open relations,
topological properties of covers of $X$ do not depend on $Y$ which is
only used as an indexed set. Namely, each open relation
$\Phi:X\ssmap Y$ is an open set in $X\times Y$ and, therefore,
$\Phi$ represents an open cover of $X$ because each preimage 
$\Phi^{-1}(y)=\Phi\cap X\times\{y\}$, $y\in Y$, is open $X$. However,
if $Y$ is a non-discrete space and $X$ is the set $Y$ endowed with the
discrete topology, then the identity relation $\Phi=\left\{\langle
  y,y\rangle: y\in Y\right\}$ is an open cover of $X$ which is not an
open relation in $X\times Y$. In fact, one can easily see that
$\Phi:X\ssmap Y$ represents an open cover of $X$ precisely when it is
an l.s.c.\ mapping (i.e.\ relation) with respect to any topology on
$Y$. \medskip

In the next section, we will prove the following theorem.

\begin{theorem}
  \label{theorem-Count-Par-v2:1}
  For a space $X$, the following are equivalent\textup{:}
  \begin{enumerate}
  \item\label{item:OG-Paracompact-v20:1} $X$ is $\tau$-paracompact and
    normal.
  \item\label{item:OG-Paracompact-v20:2} If $E$ is a topological
    vector space with $d(E)\leq\tau$, then each convex-valued open
    relation ${\Phi:X\ssmap E}$ has a continuous selection.
  \end{enumerate}
\end{theorem}

Theorem \ref{theorem-Count-Par-v2:1} is well known in the setting of
closed-convex-valued l.s.c.\ mappings, see Michael's characterisation
of paracompact spaces in \cite[Theorem 3.2$''$]{michael:56a} and
Ishii's refinement for $\tau$-paracompact spaces in \cite[Theorem
1]{MR0149440}. Each open relation is both l.s.c.\ and open-valued, but
\ref{item:OG-Paracompact-v20:2} of Theorem
\ref{theorem-Count-Par-v2:1} is not true if $\Phi$ is only assumed to
be l.s.c.\ and open-convex-valued, see \cite[Example
6.3]{michael:56a}. However, it remains valid if the range $E$ is
assumed to be a normed space and l.s.c.\ is strengthened to
``metric''-l.s.c., see Corollary \ref{corollary-Open-Rel-v9:1}. On the
other hand, most of the familiar selection theorems for l.s.c.\
mappings are implicitly based on open relations. In Section
\ref{sec:select-l.s.c.-relat}, we will explicitly illustrate this
relationship showing how open relations are used for the construction
of continuous selections for l.s.c.\ mappings and vise versa, see
Theorem \ref{theorem-Open-Rel-v4:2}. The selection problem for l.s.c.\
mappings is naturally related to the extension problem. The open
relations do not have this feature, but using some extension theorems
we can still continuously extend partial selections of such relations,
see Corollary \ref{corollary-Open-Rel-v5:1} (based on Dowker's
extension theorem \cite{dowker:52}) and Corollary
\ref{corollary-Open-Rel-v5:2} (based on Dugundji's extension theorem
\cite{dugundji:51}).\medskip

Open relations are very efficient to obtain several classical results
for paracompact spaces from a common point of view, and simplify
significantly their proofs. Most of these applications follow easily
from Theorem \ref{theorem-Count-Par-v2:1} by constructing open
relations in various function spaces. Briefly, in Section
\ref{sec:select-insert-prod} we will show that the Dowker-Kat\v{e}tov
insertion characterisation of countably paracompact normal spaces
\cite{dowker:51,katetov:51,MR0060211} is identical to a similar
selection characterisation of these spaces obtained by Michael in
\cite[Theorem 3.1$''$]{michael:56a}, see Theorem
\ref{theorem-Count-Par-v14:1}. In Section
\ref{sec:selections-products}, we will give a very simple proof of the
classical result of Morita \cite[Theorem 2.4]{morita:60} about
products of $\tau$-paracompact normal spaces with the Tychonoff cube
of $\tau$-many closed unit intervals, see Theorems
\ref{theorem-Open-Rel-vgg-2nd:1} and \ref{theorem-prod-dis-v6:2}.

\section{Selections and Open Relations}
\label{sec:select-open-graph}

In this section, we will prove Theorem
\ref{theorem-Count-Par-v2:1}. The following property is a part of this
proof, and will be also used in other proofs. 

\begin{proposition}
  \label{proposition-OG-Paracompact-v21:1}
  Let $X$ be a space such that every convex-valued open relation
  $\Phi:X\ssmap \R$ has a continuous selection. Then $X$ is normal. 
\end{proposition}

\begin{proof}
  Take disjoint closed sets $A,B\subset X$ and define a convex-valued
  open relation $\Phi:X\ssmap \R$ by $ \Phi(x)=(-\infty,-1)$ if
  $x\in A$; $ \Phi(x)=(1,+\infty)$ if $x\in B$; and ${\Phi(x)=\R}$
  otherwise. Evidently, $A$ and $B$ are completely separated by any
  continuous selection for $\Phi$. Therefore, $X$ is normal.
\end{proof}

The \emph{cozero set}, or the \emph{set-theoretic support}, of a
function $\xi:X\to \R$ is the set $\coz(\xi)=\{x\in X:\xi(x)\neq 0\}$.
A collection $\xi_\alpha:X\to [0,1]$, $\alpha\in \mathscr{A}$, of
continuous functions on a space $X$ is a \emph{partition of unity} if
$\sum_{\alpha\in \mathscr{A}}\xi_\alpha(x)=1$, for each $x\in X$. A
partition of unity $\{\xi_\alpha:\alpha\in \mathscr{A}\}$ is
\emph{locally finite} if $\{\coz(\xi_\alpha): \alpha\in \mathscr{A}\}$
is a locally finite cover of $X$, and it is \emph{index-subordinated}
to a cover $\Omega:X\ssmap \mathscr{A}$ of $X$ if
$\coz(\xi_\alpha)\subset \Omega^{-1}(\alpha)$ for every
$\alpha\in \mathscr{A}$.  The following theorem is well known, it is a
consequence of Urysohn's characterisation of normality
\cite{MR1512258} and the Lefschetz lemma \cite{MR0007093}.

\begin{theorem}
  \label{theorem-Open-Rel-v4:1}
  Every locally finite open cover of a normal space has a partition of
  unity index-subordinated to it.
\end{theorem}                                        

For a metric $\rho$ on $Y$ and $\varepsilon > 0$, we will use
$\mathbf{O}_\varepsilon(p)=\{y\in Y:\rho(y,p)<\varepsilon\}$ for the
\emph{open $\varepsilon$-ball} centred at a point $p\in Y$.  For a
space $T$, let $C_0(T)$ be the vector space of all continuous
functions $y:T\to \R$ such that the set
$\{t\in T: |y(t)|\geq \varepsilon\}$ is compact for each
$\varepsilon>0$. The linear operations in $C_0(T)$ are defined
pointwise and it is equipped with the sup-norm
$\|y\|=\sup_{t\in T}|y(t)|$, $y\in C_0(T)$.

\begin{proof}[Proof of Theorem \ref{theorem-Count-Par-v2:1}]
  Let $X$ be a $\tau$-paracompact normal space, $E$ be a topological
  vector space with $d(E)\leq \tau$, and $\Phi:X\ssmap E$ be a
  convex-valued open relation. To construct a continuous selection
  $f:X\to E$ for $\Phi$, take a dense set $\mathscr{A}\subset E$ with
  $|\mathscr{A}|\leq \tau$, and set
  $\Omega=\Phi\cap X\times\mathscr{A}$. Then
  $\Omega:X\ssmap \mathscr{A}$ is an open cover of $X$ (being an open
  relation), because $X\times\mathscr{A}$ is dense in $X\times E$ and
  $\Phi\subset X\times E$ is open.  Since $|\mathscr{A}|\leq\tau$ and
  $X$ is $\tau$-paracompact and normal, by Theorem
  \ref{theorem-Open-Rel-v4:1}, $X$ has a locally finite partition of
  unity $\{\xi_\alpha:\alpha\in \mathscr{A}\}$ which is
  index-subordinated to $\Omega:X\ssmap \mathscr{A}$, i.e.\ with the
  property that $\xi_\alpha(x)\neq 0$ implies
  $\alpha\in\Omega(x)\subset \Phi(x)$. We can now set
  $f(x)=\sum_{\alpha\in\mathscr{A}}\xi_\alpha(x)\cdot \alpha$, and it
  is trivial to check that this $f:X\to E$ is as required because the
  vector operations in $E$ are continuous and $\Phi$ is
  convex-valued. \smallskip

  Conversely, let $X$ be as in \ref{item:OG-Paracompact-v20:2} of
  Theorem \ref{theorem-Count-Par-v2:1}. Then by Proposition
  \ref{proposition-OG-Paracompact-v21:1}, it is normal. To show that
  $X$ is also $\tau$-paracompact, we argue in a similar way as in the
  proof of \cite[Theorem 5.7]{MR1961298}. Namely, take an
  \emph{increasing} open cover $\Omega:X\ssmap \tau$ of $X$, i.e.\
  such that $\Omega^{-1}(\alpha)\subset \Omega^{-1}(\beta)$ for every
  $\alpha<\beta<\tau$. The proof is now based on the space
  $C_0(\tau)$, where the initial ordinal $\tau$ is equipped with the
  usual order topology. Since the ordinal space $\tau$ is locally
  compact, we have that $w(C_0(\tau))\leq \tau$. In this setting, for
  each $x\in X$, let
  \begin{equation}
    \label{eq:OG-Paracompact-v21:1}
    \alpha(x)=\min\Omega(x)\quad \text{and}\quad 
    \Phi(x)=\left\{y\in C_0(\tau): \min_{\alpha\leq \alpha(x)}
      y(\alpha)>2\right\}. 
  \end{equation}
  Then $\Phi(x)\neq \emptyset$ because the (compact) set
  $\{\alpha<\tau: \alpha\leq \alpha(x)\}$ is clopen in $\tau$.
  Moreover, $\Phi(x)$ is convex because for $y_1,y_2\in \Phi(x)$ there
  is $\delta>2$ such that
  $y_1(\alpha),y_2(\alpha)\in [\delta,+\infty)$ for every
  $\alpha\leq \alpha(x)$. The resulting relation
  $\Phi:X\ssmap C_0(\tau)$ is also open. Indeed, take
  $\langle x_0,y_0\rangle\in \Phi$ and set
  ${\varepsilon=\min_{\alpha\leq\alpha(x_0)}\frac{y_0(\alpha)-2}2>0}$.
  Then $\mathbf{O}_\varepsilon(y_0)\subset\Phi(x_0)$.  Furthermore, if
  $\langle x,y\rangle\in \Omega^{-1}(\alpha(x_0))\times
  \mathbf{O}_\varepsilon(y_0)$, then \eqref{eq:OG-Paracompact-v21:1}
  implies that $y\in \Phi(x)$ because $\alpha(x)\leq \alpha(x_0)$ and
  $y\in \Phi(x_0)$, so 
  $\Omega^{-1}(\alpha(x_0))\times \mathbf{O}_\varepsilon(y_0)\subset
  \Phi$.  Thus, by hypothesis, $\Phi$ has a continuous selection
  ${f:X\to C_0(\tau)}$. Since $C_0(\tau)$ is paracompact (being
  metrizable) and $f$ is continuous, $X$ has a locally finite open
  cover $\mathscr{V}$ such that $\diam f(V)\leq 1$ for every
  $V\in \mathscr{V}$. As in the proof of \cite[Theorem
  5.7]{MR1961298}, this implies that $\mathscr{V}$ is a refinement of
  the cover $\Omega:X\ssmap \tau$. Briefly, for fixed
  $x\in V\in \mathscr{V}$, there is $\alpha<\tau$ with
  $|f(x)(\alpha)|<1$. For this $\alpha<\tau$, we now have that
  $V\subset \Omega^{-1}(\alpha)$. Indeed, if there is
  $p\in V\setminus \Omega^{-1}(\alpha)$, then
  \eqref{eq:OG-Paracompact-v21:1} implies that $\alpha(p)>\alpha$. For
  the same reason, $f(p)(\alpha)>2$ because $f(p)\in
  \Phi(p)$. Accordingly,
  $\|f(x)-f(p)\|\geq |f(x)(\alpha)-f(p)(\alpha)|>1$ which is clearly
  impossible because $\diam f(V)\leq 1$. Thus, $\mathscr{V}$ refines
  $\Omega:X\ssmap \tau$ and by a result of Mack \cite[Theorem
  5]{MR0211382}, $X$ is $\tau$-paracompact as well.
\end{proof}

For spaces $X$ and $Y$, a relation $\Phi:X\ssmap Y$ is \emph{lower
  semi-continuous}, or l.s.c., if the set $\Phi^{-1}[U]$ is open in
$X$ for every open $U\subset Y$. Each open cover $\Phi:X\ssmap Y$ of
$X$ is l.s.c.\ because $\Phi^{-1}[U]=\bigcup_{y\in U}\Phi^{-1}(y)$ is
open for every $U\subset Y$, but the converse is not true. For a
metric space $(Y,\rho)$, a relation $\Phi:X\ssmap Y$ is
\emph{$\rho$-l.s.c.}\ if for every $\varepsilon >0$, each point
$p\in X$ is contained in an open set $V\subset X$ such that
$\Phi(p)\subset \mathbf{O}_\varepsilon( \Phi ( x))$ for every
$x\in V$.  Each $\rho$-l.s.c.\ relation is l.s.c.\ but, again, the
converse fails.\medskip

As mentioned in the Introduction, \ref{item:OG-Paracompact-v20:2} of
Theorem \ref{theorem-Count-Par-v2:1} is not true if $\Phi$ is only
assumed to be l.s.c.\ and open-convex-valued. However, it remains
valid for $\rho$-l.s.c.\ open-convex-valued relations because such
relations are open. This is based on a property obtained in
\cite[Lemma 8.6]{michael:57} and credited to V. L. Klee.

\begin{corollary}
  \label{corollary-Open-Rel-v9:1}
  Let $X$ be a $\tau$-paracompact normal space, $E$ be a normed space
  with $d(E)\leq\tau$, and $\rho$ be the metric on $E$ generated by
  the norm of $E$.  Then each open-convex-valued $\rho$-l.s.c.\
  relation ${\Phi:X\ssmap E}$ has a continuous selection.
\end{corollary}

\begin{proof}
  Take $\langle p,q\rangle\in \Phi\subset X\times E$. Since $\Phi(p)$
  is open, there exists $\delta>0$ such that
  $\mathbf{O}_{2\delta}(q)\subset \Phi(p)$. Moreover, since $\Phi$ is
  $\rho$-l.s.c., $p$ is contained in an open set $V\subset X$ such
  that $\Phi(p)\subset \mathbf{O}_{\frac\delta2}(\Phi(x))$ for every
  $x\in V$. This implies that
  $V\times \mathbf{O}_\delta(q)\subset \Phi$. Indeed, if
  $\langle x,y\rangle\in V\times \mathbf{O}_\delta(q)$, then
  $\mathbf{O}_\delta(y)\subset \mathbf{O}_{2\delta}(q)\subset
  \Phi(p)\subset \mathbf{O}_{\frac\delta2}(\Phi(x))$. Hence, by
  \cite[Lemma 8.6]{michael:57}, $y\in \Phi(x)$. Thus, $\Phi$ is an
  open relation and by Theorem \ref{theorem-Count-Par-v2:1}, it has a
  continuous selection.
\end{proof}

Regarding the proper place of Corollary \ref{corollary-Open-Rel-v9:1},
the interested reader is referred to \cite[Theorem 8.5]{michael:57}
where Michael showed that for a space $X$ and a normed space $E$, each
open-convex-valued $\rho$-continuous relation $\Phi:X\ssmap E$ has a
continuous selection. Here, $\rho$-continuity is meant with respect to
the \emph{Hausdorff distance} generated by the metric $\rho$ on $E$
corresponding to the norm of $E$. As he remarked in \cite{michael:57},
the role of $\rho$-continuity is to reduce this selection problem to a
paracompact domain $X$. However, each $\rho$-continuous relation is
$\rho$-l.s.c. Hence, Corollary \ref{corollary-Open-Rel-v9:1} gives a
natural explanation for this selection property.

\section{Selections and L.s.c.\ Relations}
\label{sec:select-l.s.c.-relat}

In the setting of a metrizable range, each closed-valued l.s.c.\
mapping is a countable intersection of open relations or, in other
words, a $G_\delta$-relation. This follows from the following
considerations. For $\Phi\subset X\times Y$ and
$\Psi\subset Y\times Z$, the \emph{composite} relation
$\Psi\circ \Phi\subset X\times Z$ is defined~by
\[
\Psi\circ \Phi=\big\{\langle x,z\rangle \in X\times Z: \langle
x,y\rangle \in \Phi\ \text{and}\ \langle y,z\rangle\in\Psi\ \text{for
  some $y\in Y$}\big\}.
\]
Evidently, $\Psi\circ \Phi(x)=\Psi[\Phi(x)]$ for each $x\in X$. If
$V\times\{q\}\subset \Phi$ and $\{q\}\times W\subset \Psi$ for some
$V\subset X$, $q\in Y$ and $W\subset Z$, then
$V\times W\subset \Psi\circ \Phi$. Thus, the composition of open
relations is also an open relation. In fact, this is true in the
following more general situation.

\begin{proposition}
  \label{proposition-Open-Rel-v4:1}
  Let $X$, $Y$ and $Z$ be spaces, $\Phi:X\ssmap Y$ be l.s.c.\ and
  $\Psi:Y\ssmap Z$ be an open relation. Then the composite
  relation $\Psi\circ\Phi:X\ssmap Z$ is also open.
\end{proposition}

\begin{proof}
  Follows from the fact that
  $\Phi^{-1}[U]\times V\subset \Psi\circ \Phi$, whenever
  $U\times V\subset \Psi$.
\end{proof} 

For a metric space $(Y,\rho)$ and $\varepsilon>0$, the open
$\varepsilon$-balls form the open symmetric relation
$\mathbf{O}_\varepsilon=\{\langle y,z\rangle:
d(y,z)<\varepsilon\}\subset Y\times Y$ in the sense that each
point-image $\mathbf{O}_\varepsilon(y)$ by $\mathbf{O}_\varepsilon$ is
just the open $\varepsilon$-ball centred at $y\in Y$.  For a relation
$\Phi:X\ssmap Y$, we will use
$\mathbf{O}_\varepsilon[\Phi]=\mathbf{O}_\varepsilon\circ \Phi:X\ssmap
Y$ to denote the composite relation. Evidently, each point-image
$\mathbf{O}_\varepsilon[\Phi](x)=\bigcup_{y\in
  \Phi(x)}\mathbf{O}_\varepsilon(y)$ is the
$\varepsilon$-neighbourhood of $\Phi(x)$. Moreover, $\Phi:X\ssmap Y$
is closed-valued precisely when
$\Phi=\bigcap_{\varepsilon>0}\mathbf{O}_\varepsilon[\Phi]$. Another
benefit of these considerations is the following simple
characterisation of l.s.c.\ mappings in terms of open relations whose
prototype can be found in \cite[Propositions 2.1 and
6.9]{gutev:05}. Intuitively, it states that $\Phi:X\ssmap Y$ is
l.s.c.\ precisely when it has a special ``base'' of open relations.

\begin{proposition}
  \label{proposition-Open-Rel-vgg:1}
  Let $X$ be a space and $(Y,\rho)$ be a metric space. Then for a
  relation $\Phi:X\ssmap Y$, the following are equivalent\textup{:}
  \begin{enumerate}
  \item\label{item:Open-Rel-vgg:1} $\Phi$ is l.s.c.\
  \item\label{item:Open-Rel-vgg:2} $\mathbf{O}_\varepsilon[\Phi]$ is
    an open relation for every $\varepsilon>0$.
  \item\label{item:Open-Rel-vgg:3} There are open relations
    $\Omega_n:X\ssmap Y$ with
    ${\Phi\subset \Omega_n\subset \mathbf{O}_{\frac1n}[\Phi]}$,
    $n\in\N$.
  \end{enumerate}
\end{proposition}

\begin{proof}
  By Proposition \ref{proposition-Open-Rel-v4:1}, it only suffices to
  show that \ref{item:Open-Rel-vgg:3} implies
  \ref{item:Open-Rel-vgg:1}. So, let $\Omega_n:X\ssmap Y$, $n\in \N$,
  be as in \ref{item:Open-Rel-vgg:3}. Also, let $U\subset Y$ be open,
  $p\in \Phi^{-1}[U]$ and $q\in \Phi(p)\cap U$. Then
  $\mathbf{O}_{\frac2n}(q)\subset U$ for some $n\in \N$.  Moreover,
  since $\langle p,q\rangle\in \Phi\subset \Omega_n$, it follows that
  $\langle p,q\rangle\in V\times \mathbf{O}_{\delta}(q)\subset
  \Omega_n$ for some open set $V\subset X$ and $0<\delta<\frac1n$. We
  now have that $V\subset \Phi^{-1}[U]$. Indeed, for $x\in V$ and
  $y\in \mathbf{O}_{\delta}(q)$, we get that
  $y\in \Omega_n(x)\subset \mathbf{O}_{\frac1n}[\Phi](x)$ and
  $\mathbf{O}_{\frac1n}(y)\subset \mathbf{O}_{\delta+\frac1n}(q)
  \subset \mathbf{O}_{\frac2n}(q)\subset U$. Accordingly,
  $x\in \Phi^{-1}[U]$ because
  $\emptyset\neq \mathbf{O}_{\frac1n}(y)\cap \Phi(x)\subset
  \Phi(x)\cap U$. Thus, $\Phi$ is l.s.c.
\end{proof}

We shall say that a space $X$ is \emph{$\tau$-metacompact} if every
open cover of $X$ of cardinality $\leq\tau$, has a point-finite open
refinement. A space $X$ is \emph{metacompact} (called also
\emph{weakly paracompact}) if it is $\tau$-metacompact for every
$\tau$. These properties can be expressed in terms of indexed
refinements, see \cite[Lemma 1.4]{MR644284}. Namely, if $X$ is
$\tau$-metacompact and $|\mathscr{A}|\leq\tau$, then every open cover
$\Omega:X\ssmap \mathscr{A}$ of $X$ admits a point-finite open cover
$\varphi:X\ssmap \mathscr{A}$ such that
$\varphi^{-1}(\alpha)\subset \Omega^{-1}(\alpha)$,
$\alpha\in \mathscr{A}$, or equivalently $\varphi\subset \Omega$. In
this case, $\varphi$ is called an \emph{indexed refinement} (or
\emph{shrinking}) of $\Omega$.  For set-valued mappings, the condition
$\varphi\subset \Omega$ is well known and means that $\varphi$ is a
\emph{set-valued selection} (or a \emph{set-selection}, or
\emph{multi-selection}) for $\Omega$.\medskip

A metric $d$ on a vector space $E$ is \emph{translation-invariant}, or
simply \emph{invariant}, if $d(u+w,v+w)=d(u,v)$ for all $u,v,w\in
E$. A topological vector space $E$ is called \emph{Fr\'echet} if it is
locally convex and its topology is generated by a complete invariant
metric $d$. It is well known that each Fr\'echet space $E$ has a
complete invariant metric $d$ such that all open balls
$\mathbf{O}_\varepsilon(u)$, for $\varepsilon>0$ and $u\in E$, are
convex subsets of $E$, see \cite[Theorem 1.24]{MR1157815}. We now have
the following theorem.

\begin{theorem}
  \label{theorem-Open-Rel-v4:2}
  Let $X$ be a $\tau$-metacompact space and $E$ be a Fr\'echet space
  with ${d(E)\leq\tau}$. Then the following are equivalent\textup{:}
  \begin{enumerate}
  \item\label{item:Open-Rel-v4:1} Each convex-valued open relation
    ${\Theta:X\ssmap E}$ has a continuous selection.
  \item\label{item:Open-Rel-v4:2} Each closed-convex-valued
    l.s.c.\ relation $\Phi:X\ssmap E$ has a continuous selection.
  \end{enumerate}
\end{theorem}

\begin{proof}
  Assume that \ref{item:Open-Rel-v4:1} holds, and $\Phi:X\ssmap E$
  is as in \ref{item:Open-Rel-v4:2}. Take a complete invariant metric
  $d$ on $E$ such that all open balls are convex. By Proposition
  \ref{proposition-Open-Rel-vgg:1}, the associated relation
  $\mathbf{O}_{2^{-1}}[\Phi]:X\ssmap E$ is open. It is also
  convex-valued because
  $\mathbf{O}_{2^{-1}}[\Phi](x)=\Phi(x)+\mathbf{O}_{2^{-1}}(\mathbf{0})$,
  $x\in X$, where $\mathbf{0}$ is the origin of $E$. Therefore, by
  \ref{item:Open-Rel-v4:1}, $\mathbf{O}_{2^{-1}}[\Phi]$ has a
  continuous selection $f_1:X\to E$. Then the intersection of each
  open ball $\mathbf{O}_{2^{-1}}(f_1(x))= \mathbf{O}_{2^{-1}}[f_1](x)$
  with $\Phi(x)$ is nonempty, so the resulting relation
  $\mathbf{O}_{2^{-1}}[f_1]\cap \mathbf{O}_{2^{-2}}[\Phi]:X\ssmap E$
  is also open and convex-valued. Hence, for the same reason, it has a
  continuous selection $f_2:X\to E$. Thus, by induction, there exists
  a sequence $f_n:X\to E$, $n\in\N$, of continuous maps such that
  $f_{n+1}$ is a selection for
  $\mathbf{O}_{2^{-n}}[f_n]\cap \mathbf{O}_{2^{-(n+1)}}[\Phi]$ for
  each $n\in\N$.  Accordingly, $\{f_n\}$ is uniformly convergent to a
  continuous map $f:X\to E$ because $d$ is a complete metric on
  $E$. Moreover, $f$ is a selection for $\Phi$ because $\Phi$ is
  closed-valued and $d(f (x), \Phi(x)) = 0$ for all
  $x\in X$.\smallskip

  Conversely, let \ref{item:Open-Rel-v4:2} hold, $\Theta:X\ssmap E$ be
  a convex-valued open relation, and $\mathscr{A}\subset E$ be a dense
  set with $|\mathscr{A}|\leq \tau$. As in the proof of Theorem
  \ref{theorem-Count-Par-v2:1}, the intersection
  $\Omega=\Theta\cap X\times \mathscr{A}$ is an open cover
  $\Omega:X\ssmap \mathscr{A}$ of $X$. Since $|\mathscr{A}|\leq \tau$
  and $X$ is $\tau$-metacompact, there exists an open and point-finite
  cover $\varphi: X\ssmap \mathscr{A}$ of $X$ with
  $\varphi\subset \Omega\subset\Theta$. Then
  $\varphi:X\ssmap \mathscr{A}\subset E$ is l.s.c., and each
  $\varphi(x)$ is a finite subset of $\Theta(x)$. We can now set
  $\Phi(x)$ to be the convex hull of $\varphi(x)$, so that $\Phi(x)$
  is a closed convex subset of $\Theta(x)$ for each $x\in X$. Moreover,
  by \cite[Proposition 2.6]{michael:56a}, $\Phi:X\ssmap E$ is also
  l.s.c. Hence, by \ref{item:Open-Rel-v4:2}, $\Phi$ has a continuous
  selection and, therefore, $\Theta$ has a continuous selection as well.
\end{proof}

Theorems \ref{theorem-Count-Par-v2:1} and
\ref{theorem-Open-Rel-v4:2} imply the following result obtained by
Michael in \cite[Theorem 1.2]{michael:66}, for an alternative proof
see \cite[Theorem 5.1]{zbMATH05275704}.

\begin{corollary}
  \label{corollary-Open-Rel-v4:2}
  If $X$ is a $\tau$-paracompact normal space and $E$ is a Fr\'echet
  space with $d(E)\leq\tau$, then every closed-convex-valued
  l.s.c.\ relation $\Phi:X\ssmap E$ has a continuous selection.
\end{corollary}

The selection problem for l.s.c.\ mappings is a natural generalisation
of the extension problem, see \cite[Examples 1.3 and
1.3$^*$]{michael:56a}. In contrast, continuous selections for open
relations do not imply any one of the familiar extensions
theorems. Briefly, as in \cite{michael:56a}, for spaces $X$ and $Y$, a
closed set $A\subset X$, an open relation $\Phi:X\ssmap Y$ and a
continuous selection $g:A\to Y$ for $\Phi\uhr A$, define
$\varphi:X\ssmap Y$ by $\varphi(x)=\{g(x)\}$ if $x\in A$ and
$\varphi(x)=\Phi(x)$ otherwise. Then $\varphi$ is l.s.c., but is not
necessary an open relation. However, some of the extension theorems
are very useful to obtain continuous extensions of partial selections
for open relations.

\begin{corollary}
  \label{corollary-Open-Rel-v5:1}
  Let $X$ be a $\tau$-paracompact normal space, $A\subset X$ be a
  closed set, $E$ be a Fr\'echet space with $d(E)\leq\tau$, and
  $\Phi:X\ssmap E$ be a convex-valued open relation. Then every
  continuous selection $g:A\to E$ for $\Phi\uhr A$ can be extended
  to a continuous selection for $\Phi$.
\end{corollary}

\begin{proof}
  By Dowker's extension theorem \cite{dowker:52}, see also Corollary
  \ref{corollary-Open-Rel-v4:2}, $g$ can be extended to a continuous
  map $\tilde{g}:X\to E$. Since $\Phi$ is an open relation and
  $g=\tilde{g}\uhr A$ is a selection $\Phi\uhr A$, i.e.\
  $\tilde{g}\uhr A \subset \Phi\uhr A$, there is an open set
  $U\subset X$ such that $A\subset U$ and
  $\tilde{g}\uhr U \subset \Phi\uhr U$. Take a continuous function
  $\eta:X\to [0,1]$ such that $A\subset \eta^{-1}(1)$ and
  $X\setminus U\subset \eta^{-1}(0)$. Next, using Theorem
  \ref{theorem-Count-Par-v2:1}, take a continuous selection $h:X\to E$
  for $\Phi$. Finally, define another continuous map $f:X\to E$ by
  $f(x)=\eta(x)\cdot \tilde{g}(x)+(1-\eta(x))\cdot h(x)$, $x\in
  X$. Evidently, $f$ is a continuous selection for $\Phi$ with
  $f\uhr A=g$.
\end{proof}

Applying the same proof as in Corollary \ref{corollary-Open-Rel-v5:1},
but now using Dugundji's extension theorem \cite{dugundji:51} instead
of Dowker's one \cite{dowker:52}, we also get the following selection
extension result.

\begin{corollary}
  \label{corollary-Open-Rel-v5:2}
  Let $X$ be a metrizable space, $A\subset X$ be a
  closed set, $E$ be a locally convex topological vector space, and
  $\Phi:X\ssmap E$ be a convex-valued open relation. Then every
  continuous selection $g:A\to E$ for $\Phi\uhr A$ can be extended
  to a continuous selection for $\Phi$.  
\end{corollary}

\section{Insertions and Open Relations}
\label{sec:select-insert-prod}

A function $\xi:X\to \R$ is \emph{lower} (\emph{upper})
\emph{semicontinuous} if the set
\[
\{x\in X:\xi(x)>r\}\quad \text{(respectively, $\{x\in X:\xi(x)<r\}$)}
\]
is open in $X$ for every $r\in \R$. For functions $\xi,\eta:X\to \R$,
we write $\xi< \eta$ to express that $\xi(x)<\eta(x)$ for every
$x\in X$.\medskip

The following theorem was obtained by Dowker \cite{dowker:51}, the
insertion property was also established by Kat\v{e}tov
\cite{katetov:51,MR0060211}.

\begin{theorem}[\cite{dowker:51,katetov:51,MR0060211}]
  \label{theorem-Count-Par-v1:1}
  The following properties of a space X are
  equivalent\textup{:}
  \begin{enumerate}
  \item\label{item:OG-Paracompact-v13:1} $X$ is countably paracompact
    and normal.
  \item\label{item:OG-Paracompact-v13:2} If $\xi:X\to \R$ is upper
    semicontinuous, $\eta:X\to \R$ is lower semicontinuous and
    $\xi<\eta$, then there is a continuous function $f:X\to \R$
    with $\xi < f < \eta$.
  \item\label{item:OG-Paracompact-v13:3} The topological product
    $X\times[0,1]$ is normal.
  \end{enumerate}
\end{theorem}

The following selection theorem was obtained by Michael \cite[Theorem
3.1$''$]{michael:56a}, it was motivated as a selection interpretation
of Theorem \ref{theorem-Count-Par-v1:1}.

\begin{theorem}[\cite{michael:56a}]
\label{theorem-Count-Par-v1:2}
  The following properties of a space $X$ are
  equivalent\textup{:}
  \begin{enumerate}
  \item\label{item:OG-Paracompact-v13:4} $X$ is countably
    paracompact and normal.
  \item\label{item:OG-Paracompact-v13:5} If $E$ is a separable Banach
    space, then every closed-convex-valued l.s.c.\ relation
    $\Phi:X\ssmap E$ has a continuous selection.
  \item\label{item:OG-Paracompact-v13:6} Every closed-convex-valued
    l.s.c.\ relation $\Phi:X\ssmap \R$ has a continuous selection.
  \end{enumerate}
\end{theorem}

The Dowker-Kat\v{e}tov insertion property in
\ref{item:OG-Paracompact-v13:2} of Theorem
\ref{theorem-Count-Par-v1:1} has the following natural interpretation
in terms of selections for convex-valued open relations.

\begin{proposition}
  \label{proposition-Count-Par-v2:1}
  For a space $X$, the following are equivalent\textup{:}
  \begin{enumerate}
  \item\label{item:Open-Rel-vgg-4th:1} Every convex-valued open
    relation $\Phi:X\ssmap \R$ has a continuous selection.
  \item\label{item:Open-Rel-vgg-4th:2} If $\xi:X\to \R$ is upper
    semicontinuous, $\eta:X\to \R$ is lower semicontinuous and
    $\xi<\eta$, then there is a continuous function $f:X\to \R$
    with $\xi < f < \eta$. 
  \end{enumerate}
\end{proposition}

\begin{proof}
  Using the order preserving homeomorphism
  $\frac{t}{1+|t|}:\R\to (-1,1)$, the proof is reduced to showing that
  a relation $\Phi:X\ssmap (-1,1)$ is open and convex-valued if and
  only if there are functions $\xi,\eta:X\to [-1,1]$ such that $\xi$
  is upper semicontinuous, $\eta$ is lower semicontinuous and
  $\Phi(x)=\big(\xi(x),\eta(x)\big)$, $x\in X$.  For
  $\Phi:X\ssmap (-1,1)$, set $\xi(x)=\inf\Phi(x)$ and
  $\eta(x)=\sup\Phi(x)$ for ${x\in X}$. Evidently,
  $\xi,\eta:X\to [-1,1]$ and $\xi\leq \eta$. Moreover, $\Phi$ is
  open-convex-valued precisely when $\xi<\eta$ and
  $\Phi(x)=\big(\xi(x),\eta(x)\big)$, $x\in X$. Furthermore, if $\Phi$
  is open-convex-valued, $p\in V\subset X$ and $s,t\in\R$ with
  $\xi(p)<s<t<\eta(p)$, then we have that $V\times [s,t]\subset \Phi$
  if and only if $\xi(x)< s<t<\eta(x)$ for every $x\in V$. The proof
  is complete.
\end{proof}

In view of Theorem \ref{theorem-Open-Rel-v4:2} and Proposition
\ref{proposition-Count-Par-v2:1}, we now have that the
Dowker-Kat\v{e}tov insertion property in Theorem
\ref{theorem-Count-Par-v1:1} and the selection one in Theorem
\ref{theorem-Count-Par-v1:2} are essentially equivalent to the
following selection property for open relations.

\begin{theorem}
  \label{theorem-Count-Par-v14:1}
  For a space $X$, the following are equivalent\textup{:}
  \begin{enumerate}
  \item\label{item:OG-Paracompact-v21:1} $X$ is countably
    paracompact and normal.
  \item\label{item:OG-Paracompact-v21:2} If $E$ is a separable
    topological vector space, then each convex-valued open relation
    $\Phi:X\ssmap E$ has a continuous selection.
  \item\label{item:OG-Paracompact-v21:3} Every convex-valued open
    relation $\Phi:X\ssmap \R$ has a continuous selection.
  \end{enumerate}
\end{theorem}

\begin{proof}
  The implications
  \ref{item:OG-Paracompact-v21:1}$\implies
  $\ref{item:OG-Paracompact-v21:2}$\implies
  $\ref{item:OG-Paracompact-v21:3} are the special case of
  $\tau=\omega$ in Theorem \ref{theorem-Count-Par-v2:1}. If $X$ is as
  in \ref{item:OG-Paracompact-v21:3}, by Proposition
  \ref{proposition-OG-Paracompact-v21:1}, it is normal.  To show that
  it is also countably paracompact, take an open cover
  $\Omega:X\ssmap \N$ of $X$. Next, using Proposition
  \ref{proposition-Open-Rel-v4:1} and the open relation
  $\Lambda=\{\langle k,t\rangle\in \N\times\R: k<t\}$, define an open
  relation $\Phi:X\ssmap \R$ by $\Phi=\Lambda\circ \Omega$. Evidently,
  $\Phi(x)=(\min\Omega(x),+\infty)$ for each $x\in X$. Hence, by
  \ref{item:OG-Paracompact-v21:3}, $\Phi$ has a continuous selection
  $f:X\to \R$. Since $f$ is continuous, the relation
  $\Gamma=\{\langle x,t\rangle\in X\times\R: t<f(x)\}$ is open. Then
  the intersection $\varphi=\Omega\cap \Gamma$ is an open cover
  $\varphi:X\ssmap \N$ of $X$ because $\min\Omega(x)<f(x)$, for each
  $x\in X$.  Moreover, for $p\in X$ and $t\in \R$ with $f(p)<t$, the
  set $V=f^{-1}\big((-\infty,t)\big)$ is open, $p\in V$ and
  $\varphi[V]$ is finite because $\varphi[V]\subset [0,t]$. Thus, $V$
  intersects only finitely many members of the cover
  $\{\varphi^{-1}(n):n\in \N\}$. Hence, $\varphi:X\ssmap \N$ is a
  locally finite open cover of $X$ with $\varphi\subset
  \Omega$. Accordingly, $X$ is countably paracompact as well.
\end{proof}

We conclude this section with some remarks regarding the role of open
relations in Theorem \ref{theorem-Count-Par-v14:1}. In fact, we have
the following consequence of this theorem. 

\begin{corollary}
  \label{corollary-Open-Rel-vgg:1}
    For a space $X$, the following are equivalent\textup{:}
  \begin{enumerate}
  \item\label{item:Open-Rel-vgg:4} $X$ is countably
    paracompact and normal.
  \item\label{item:Open-Rel-vgg:5} Every open-convex-valued l.s.c.\
    relation $\Phi:X\ssmap \R$ has a continuous selection.
  \end{enumerate}
\end{corollary}

\begin{proof}
  By Theorem \ref{theorem-Count-Par-v14:1}, it suffices to show that
  each open-convex-valued l.s.c.\ relation $\Phi:X\ssmap \R$ is an
  open relation. To this end, as in the proof of Proposition
  \ref{proposition-Count-Par-v2:1}, take $p\in X$ and $s,t\in \Phi(p)$
  with $s<t$. Then $(-\infty,s)\cap \Phi(p)\neq \emptyset$ because
  $\Phi(p)$ is open. Similarly,
  $(t,+\infty)\cap \Phi(p)\neq \emptyset$. Hence,
  $V=\Phi^{-1}\left[(-\infty,s)\right]\cap
  \Phi^{-1}\left[(t,+\infty)\right]$ is an open set with $p\in
  V$. Moreover, $V\times[s,t]\subset \Phi$ because $\Phi$ is
  convex-valued.
\end{proof}

Corollary \ref{corollary-Open-Rel-vgg:1} is complementary to Michael's
result in \cite[Theorem 3.1$'''$]{michael:56a} that a space $X$ is
perfectly normal precisely when each convex-valued l.s.c.\ relation
$\Phi:X\ssmap \R$ has a continuous selection. In particular, the
condition in Corollary \ref{corollary-Open-Rel-vgg:1} that $\Phi$ is
open-valued can not be dropped.

\section{Products and Open Relations}
\label{sec:selections-products}

The following generalisation of the equivalence
\ref{item:OG-Paracompact-v13:1}$\iff$\ref{item:OG-Paracompact-v13:3}
in Theorem \ref{theorem-Count-Par-v1:1} was established by Morita
\cite[Theorem 2.4]{morita:60}.

\begin{theorem}[\cite{morita:60}]
  \label{theorem-Open-Graph-v1:1}
  A space $X$ is $\tau$-paracompact and normal if and only if
  the topological product $X\times[0,1]^\tau$ is normal.
\end{theorem}

Here, we will show that Theorem \ref{theorem-Open-Graph-v1:1} is a
special selection problem for open relations and follows easily from
Theorems \ref{theorem-Count-Par-v2:1} and
\ref{theorem-Count-Par-v14:1}.  This will be based on the exponential
law that
$Z^{X\times Y}\ni g\longleftrightarrow \bm\hat{g}\in
\left(Z^Y\right)^X$, see Dugundji \cite{dugundji:66}. Namely, for each
$g:X\times Y\to Z$, the formula $\bm\hat{g}[x](y) =g(x,y)$ defines a
family of maps $\bm\hat{g}[x]:Y\to Z$ indexed by the elements of
$X$. Thus, $x\to \bm\hat{g}[x]$ is a map $\bm\hat{g}:X\to
Z^Y$. Conversely, given $\bm\hat{g}:X\to Z^Y$, the same formula
defines a map $g:X\times Y\to Z$.\medskip

For a compact space $Y$, we will use $C(Y)$ for the vector space of
all continuous functions $f:Y\to \R$ equipped with the sup-norm
$\|f\|=\sup_{y\in Y}|f(y)|$, equivalently with the compact open
topology. The following property is well known (see, for instance,
\cite[Theorems 3.1 in Chapter XII]{dugundji:66}):
\begin{equation}
  \label{eq:Open-Rel-vgg-fd:1}
  g:X\times Y\to \R\ \text{is
  continuous if and only if so is}\ \bm\hat{g}:X\to C(Y).
\end{equation}

We can now apply Theorems \ref{theorem-Count-Par-v2:1} and
\ref{theorem-Count-Par-v14:1} to get the following simple proof of
\cite[Theorem 2.2]{morita:60}.

\begin{theorem}
  \label{theorem-Open-Rel-vgg-2nd:1}
  Let $X$ be a $\tau$-paracompact normal space, and $Y$ be a compact
  space with $w(Y)\leq \tau$. Then $X\times Y$ is also normal.
\end{theorem}

\begin{proof}
  Take a convex-valued open relation $\Phi:X\times Y\ssmap \R$, and
  define another relation $\Psi\subset X\times C(Y)$ by
  \[
    \Psi[x]=\big\{f\in C(Y): f(y)\in \Phi(x,y),\ \text{for every
      $y\in Y$}\big\}.
  \]
  According to Theorem \ref{theorem-Count-Par-v14:1},
  $\Psi:X\ssmap C(Y)$ and is clearly convex-valued. It is also
  open. Indeed, if $p\in X$ and $f\in \Psi[p]$, then $\{p\}\times f$
  is a compact subset of the open set $\Phi\subset X\times
  Y\times\R$. Hence, there is an open set $V\subset X$ and $\delta>0$
  such that $p\in V$ and $V\times\mathbf{O}_\delta[f]\subset
  \Psi$. Since $d(C(Y))=w(Y)\leq \tau$ and $X$ is $\tau$-paracompact
  and normal, it follows from Theorem \ref{theorem-Count-Par-v2:1}
  that $\Psi$ has a continuous selection $\bm\hat{g}:X\to C(Y)$. Then
  by \eqref{eq:Open-Rel-vgg-fd:1}, the associated function
  $g:X\times Y\to \R$ is also continuous, and obviously it is a
  selection for $\Phi$. Therefore, by Proposition
  \ref{proposition-OG-Paracompact-v21:1}, $X\times Y$ is normal.
\end{proof}

In a similar way, we have the inverse implication of Theorem
\ref{theorem-Open-Rel-vgg-2nd:1}, in fact the complete proof of
Theorem \ref{theorem-Open-Graph-v1:1}. To this end, for an ordered set
$(\mathscr{A},\leq)$, let us recall that a cover
$\Omega:X\ssmap \mathscr{A}$ of $X$ is increasing if
$\Omega^{-1}(\alpha)\subset \Omega^{-1}(\beta)$ for every
$\alpha,\beta\in\mathscr{A}$ with $\alpha<\beta$. This is equivalent
to the property that $[\alpha,\to)\subset \Omega(x)$ for every
$x\in X$ and $\alpha\in \Omega(x)$, where $[\alpha,\to)$ is the
interval of all $\beta\in\mathscr{A}$ with $\alpha\leq \beta$.
Ordinals are the natural range for increasing relations. As before,
each ordinal $\lambda$ will be equipped with the usual topology and
order as an ordinal space. Moreover, $\mathscr{I}_\lambda$ will stand
for the isolated points of $\lambda$.  In this setting, for a cardinal
number $\tau$, we have that $\mathscr{I}_\tau$ is a discrete space
with $|\mathscr{I}_\tau|=\tau$ which will play an interesting role
regarding $\tau$-paracompactness.\medskip

For a relation $\Phi: X\ssmap Y$ between topological spaces, the
\emph{pointwise closure} $\overline{\Phi}: X\ssmap Y$ of $\Phi$
is defined by $\overline{\Phi}(x)=\overline{\Phi(x)}$ for every
$x\in X$.

\begin{proposition}
  \label{proposition-prod-dis-v6:5}
  Let $\lambda$ be an infinite ordinal and
  $\Omega:X\ssmap {\mathscr{I}_\lambda}\subset \lambda$ be an
  increasing open cover of $X$. Then
  $\overline{\Omega}: X\ssmap \lambda$ is an increasing open relation.
\end{proposition}

\begin{proof}
  In this proof, the intervals $[\alpha,\to)$ are taken in $\lambda$.
  Evidently, $\overline{\Omega}$ is increasing because
  $\overline{[\alpha,\to)\cap \mathscr{I}_\lambda}=[\alpha,\to)$ for
  every $\alpha\in \mathscr{I}_\lambda$. To see that
  $\overline{\Omega}$ is open, take
  $\langle p,\alpha\rangle\in \overline{\Omega}\setminus \Omega$.
  Then $\alpha\in \overline{\Omega(p)}\setminus \Omega(p)$ and there
  is $\beta\in \mathscr{I}_\lambda$ with $\beta<\alpha$ and
  $\langle p,\beta\rangle \in \Omega$. Accordingly,
  $\langle p,\beta\rangle\in \Omega^{-1}(\beta)\times\{\beta\}\subset
  \Omega\subset\overline{\Omega}$ and $\Omega^{-1}(\beta)$ is open in
  $X$ because $\Omega$ is an open cover of $X$. Since
  $\beta\in \mathscr{I}_\lambda$ is an isolated point of $\lambda$ and
  $\overline{\Omega}$ is increasing, $\langle p,\alpha\rangle$ is
  contained in the open set
  $\Omega^{-1}(\beta)\times [\beta,\to)\subset \overline{\Omega}$.
\end{proof}

The following theorem gives a very simple proof of the other
implication of Theorem \ref{theorem-Open-Graph-v1:1}, compare with
\cite[Lemma 2.5]{morita:60} and the subsequent proof of the ``if''
part of \cite[Theorem 2.4]{morita:60}.

\begin{theorem}
  \label{theorem-prod-dis-v6:2}
  Let $X$ be a space such that $X\times(\tau+1)$ is normal for some
  infinite cardinal $\tau$. Then $X$ is $\tau$-paracompact.
\end{theorem}

\begin{proof}
  Take an increasing open cover $\Omega:X\ssmap \mathscr{I}_\tau$ of
  $X$. Then $\mathscr{I}_{\tau+1}=\mathscr{I}_\tau$ and by Proposition
  \ref{proposition-prod-dis-v6:5}, the pointwise closure
  $\Lambda=\overline{\Omega}:X\ssmap \tau+1$ of
  $\Omega:X\ssmap \mathscr{I}_{\tau+1}$ remains open and increasing in
  $X\times(\tau+1)$. Moreover, $\Lambda$ contains the closed set
  $X\times\{\tau\}$ because
  $X=\bigcup_{\alpha\in \mathscr{I}_{\tau}}
  \Omega^{-1}(\alpha)\subset \Lambda^{-1}(\tau)$.  Since
  $X\times(\tau+1)$ is normal, there exists a continuous function
  $\xi:X\times(\tau+1)\to [0,1]$ such that
  $X\times\{\tau\}\subset \xi^{-1}(1)$ and $\coz(\xi)\subset
  \Lambda$. Then by \eqref{eq:Open-Rel-vgg-fd:1}, the associated
  map $\bm\hat{\xi}:X\to C(\tau+1)$ is continuous. Let
  $Z=\bm\hat{\xi}(X)$ and
  $\bm\hat{\imath}=\bm\hat{\imath}_Z:Z\hookrightarrow C(\tau+1)$ be
  the natural inclusion map $\bm\hat{\imath}(z)=z$, $z\in Z$. By the
  same token, the associated map $i:Z\times(\tau+1)\to \R$
  corresponding to $\bm\hat{\imath}:Z\hookrightarrow C(\tau+1)$ is
  also continuous. In fact, we have the following commutative
  diagrams, where $1_{\tau+1}$ is the identity of $\tau+1$.
  \begin{center}
    \begin{tikzcd}[row sep=large]
     X \arrow[d, "\bm\hat{\xi}"] \arrow[rd, "\bm\hat{\xi}"]& \\
      {Z} \arrow[r, hook, "\bm\hat{\imath}"] & C(\tau+1)
    \end{tikzcd}\quad$\implies$ \quad
      \begin{tikzcd}[row sep=large]
     X\times(\tau+1) \arrow[d, "\bm\hat{\xi}\times 1_{\tau+1}"]
     \arrow[rd, "\xi"]& \\ 
      {Z}\times(\tau+1) \arrow[r,  "i"] & {[0,1]}
    \end{tikzcd}  
  \end{center}
  
  Accordingly,
  $X\times\{\tau\}\subset \coz(\xi)=\left[\bm\hat{\xi}\times
    1_{\tau+1}\right]^{-1}(\coz(i))\subset \Lambda=\overline{\Omega}$
  and, therefore, $\Psi=\coz(i):Z\ssmap \tau+1$ is an open relation
  with $\Psi\circ \bm\hat{\xi}\subset \Lambda=\overline{\Omega}$.
  Since $Z\times \mathscr{I}_\tau$ is dense in $Z\times (\tau+1)$, we
  can define another open relation $\Phi:Z\ssmap \mathscr{I}_\tau$ by
  $\Phi= \Psi\cap Z\times\mathscr{I}_\tau$. The main advantage of this
  relation is the property that $\Phi\circ \bm\hat{\xi}\subset \Omega$
  because $\Psi\circ \bm\hat{\xi}\subset \overline{\Omega}$. We can
  now use that $Z$ is paracompact (being metrizable) and
  $\Phi:Z\ssmap \mathscr{I}_\tau$ is an open cover of $Z$, to get an
  open locally finite cover $\varphi:Z\ssmap \mathscr{I}_\tau$ of $Z$
  with $\varphi\subset \Phi$. Finally, since $\bm\hat{\xi}$ is
  continuous, $\varphi\circ \bm\hat{\xi}:X\ssmap \mathscr{I}_\tau$ is
  an open locally finite cover of $X$ such that
  $\varphi\circ \bm\hat{\xi}\subset \Phi\circ \bm\hat{\xi}\subset
  \Omega$. According to \cite[Theorem 5]{MR0211382}, $X$ is
  $\tau$-paracompact.
\end{proof}

\end{document}